\documentclass{article}
\usepackage[margin=1.2in]{geometry}
\usepackage{multirow}
\usepackage{graphicx}
\usepackage{mathrsfs}
\usepackage{amsmath}
\numberwithin{equation}{section}
\usepackage{tikz}
\usetikzlibrary{3d,calc}
\usepackage[toc,page]{appendix}
\usepackage{pifont}
\usepackage{graphicx}
\usepackage[numbers,sort&compress]{natbib}  
\bibpunct{[}{]}{,}{n}{}{;}                 
\usepackage{lipsum}
\usepackage{hyperref}
\usepackage{tikz-cd}
\usepackage{caption}
\usepackage{amsmath}
\usepackage{amssymb}
\usepackage{amsthm}
\usepackage{amssymb}
\usepackage[all]{xy}
\usepackage { tikz }
\usepackage{soul, color, xcolor}
\newtheorem{theorem}{Theorem}[section]
\newtheorem{lemma}{Lemma}[section]
\newtheorem{corollary}{Corollary}[section]
\newtheorem{definition}{Definition}[section]
\newtheorem{example}[theorem]{Example}

\title{A New Approach from Lattice of Subgroup Sets to Generalized Solvable {Extension} Formations
}

\author{
Ran Li, Long Miao\footnote{Corresponding author}, Wenxia Zhou, Yinan Chen}

\date{}

\begin{document}
\maketitle

\begin{abstract}	
{\small  
\noindent In this paper, we establish the decomposition of morphisms
from lattice of subgroup sets to generalized solvable extension formations. To achieve this, we develop a unified framework involving maximal subgroup functors, generating formation morphism and contraction-extension functors. In particular, solvability-induced sets of maximal subgroups are determined and generating formation morphism gives rise to generalized
solvable extension formations.
}
\end{abstract}

\section{Introduction}

\qquad In the book \textit{Finite Group Theory}\citep{Martin-2008}, Isaacs states the local-global principle: global information can be deduced from appropriate data about local subgroups. Focusing on Sylow p-subgroups, numerous results have been established(see\citep{Glauberman-1968},\citep{Gorenstein-1994}). 
{In addition, maximal subgroups are guaranteed to exist in any finite group and serve as key local structures.}
Craven \citep{Craven-2023} lists all maximal subgroups to study the finite simple
groups of type $F_4$, $E_6$ and twisted $E_6$ over all finite fields.\\ \\
\qquad Building on Craven’s work, we extend the framework to the level of  maximal subgroup sets and formations shown in Diagram \eqref{eq:1.1}.  
The motivation also stems from the fundamental theorem of group homomorphism shown in Diagram \eqref{eq:1.2}. 
\begin{equation}\label{eq:1.1}
\begin{tikzcd}[column sep=normal, row sep=3em]
MAX(G) \arrow[d,dashed, "f"'] \arrow[r, "u"] & F(\mathcal{G}) \\
\overline{MAX_2(G)}\arrow[ur, "g"'] &
\end{tikzcd}
\begin{tikzcd}[column sep=large, row sep=3em]
\mathcal{Z} \arrow[d,dashed, mapsto, "f"'] \arrow[r, mapsto, "u"] & \mathfrak{F}\\
f(\mathcal{Z})\arrow[ur, mapsto, "g"'] &
\end{tikzcd}
\end{equation}

\begin{equation}\label{eq:1.2}
\begin{tikzcd}[column sep=normal, row sep=3em]
G\arrow[d, "\phi"'] \arrow[r, "u"] & G'\\
G/Ker(u)\arrow[ur, "\tau"'] &
\end{tikzcd}
\begin{tikzcd}[column sep=large, row sep=3em]
g\arrow[d, mapsto, "\phi"'] \arrow[r, mapsto, "u"] & g'\\
gker(u)\arrow[ur, mapsto, "\tau"'] &
\end{tikzcd}
\end{equation}

We postpone our notation and terminology to Section 4 and will introduce, at this point, only notation
and terminology needed for Diagram \eqref{eq:1.1} and Diagram \eqref{eq:1.3}. Unless otherwise stated, all groups considered in this paper are finite. 
Let $Max(G)$ be the set of all maximal subgroups of \(G\) and $Max_2(G)$ the set of all second maximal subgroups of \(G\). We denote the power set of $Max(G)$ by $\operatorname{MAX}(G)$ and the power set of $Max_2(G)$ by $\operatorname{MAX}_2(G)$.
$\overline{MAX_2(G)}$ is defined as the collection of suitable sets of second maximal subgroups of \(G\) and $F(\mathcal{G})$ is the
set of all formations. \\

As noted by Huppert\citep{Huppert-1967} in his book, if a group $G$ belongs to the formation of all solvable groups $\hat{\mathfrak{S}}$ and $N\lhd G$, then $N\in \hat{\mathfrak{S}}$ and $G/N\in \hat{\mathfrak{S}}$. Its contrapositive statement
is a classic result of the local-global principle: non-solvability behaves in a “local-to-global” manner.
Motivated by the transitivity of non-solvability, we {consider} the following sets of maximal subgroups to explore non-solvable finite groups:
$$\mathcal{X}=\{M\lessdot G \mid  M~is~non-solvable\}.$$
Ballester-Bolinches and Ezquerro, in their book \textit{Classes of Finite Groups} \citep{Ballester-2006}, states: {Helmut Wielandt proposed giving priority after the classification to the extension of these brilliant results of the theory of finite soluble groups to the more ambitious universe of all finite
groups.} {In fact, set $\mathcal{X}$ enlarges the methods of the solvable cases to finite non-nessarily solvable, according to Wielandt's proposal.} Table 1 shows some examples about solvability of $M$ in non-solvable groups.

\begin{center}
\begin{tabular}{*{6}{c}}
\hline
$G$ & $M~$ &Solvability~of~$M$   \\
\hline
\multirow{3}{*}{$S_5$} & $S_3{:}~2$&$T$   \\
 & $S_4$   &$T$    \\
& $A_5$   &$F$    \\
\hline
\multirow{3}{*}{$A_5$} & $A_4$&$T$   \\
 & $D_{10}$   &$T$    \\
& $S_3$   &$T$    \\
\hline
\end{tabular} 
\par\small \vspace{0.05cm}\textbf{Table 1:}  Solvability of $M$.
\end{center}
 In this context,  we focus on Diagram \eqref{eq:1.3} to illustrate the decomposition of morphisms between lattice of maximal subgroup sets and generalized
solvable extension formations. Here, solid arrows indicate explicitly defined morphisms, whereas dashed arrows refer to morphisms whose existence follows from the corresponding morphism construction.
\begin{equation}\label{eq:1.3}
\begin{tikzcd}[column sep=5em, row sep=5em]
  & MAX(G) \arrow[r, dashed,"u"] \arrow[d, dashed,"f"']\arrow[dl, "\Phi"'] &  F(\mathcal{G})\\
  MAX_2(G)\arrow[r, "H"'] 
  & \overline{MAX_2(G)}
  \arrow[ur, bend left=0, "g" description] 
\end{tikzcd}
\end{equation}

The paper is organized as follows. Section 2 is about {generalized solvable extension formations} and examples.
In section 3, we define various decomposition and state their main properties.  Section 4 covers the existing results and lemmas, which will be used in the sequel. Section 5 is devoted to set $\mathcal{X}$ and its localization associated with functor $\mathcal{L}_1$ and $\mathcal{L}_2$. The interplay between these sets and formation $\mathfrak{J},~\mathfrak{F}'~\&~\mathfrak{F}'' $ are shown. It is worth mentioning that the proofs in this section highlights the subtle differences in the behavior of functor $\mathcal{L}_1$ and $\mathcal{L}_2$. Section 6 covers relative applications in this topic.

\section{Generalized
Solvable Extension Formations}
In \textit{The Theory of Classes of Groups}, Guo\citep{Guo-2000} mentioned the method of constructing local formation by the concept of formation functions. {A formation function }\citep{Doerk-1992} is a function $$f:~\mathbb{P}\rightarrow  \{\text{formations}\},$$
where 
$\mathbb{P}$ denotes the set of all primes. In particular, one takes a group function defined on the set of all simple groups:
\[
f:\{\text{simple groups}\}\longrightarrow \{\text{formations}\},
\]
assigning to each simple group a formation. In the composition satellite approach, Shemetkov\citep{Shemetkov-1997} first assigns a formation to each composition factor. In this section, we extend this approach by assigning a formation to set of simple groups, rather than just singletons. Let $\mathcal{S}$ be the collection of all finite simple groups and $P(\mathcal{S})$ be the power set of $\mathcal{S}$. Now we consider
\[
f_1: P(\mathcal{S})\longrightarrow \{\text{formations}\},
\]
$$s\mapsto\mathfrak{F}$$
s.t. $\mathfrak{F}$ consists of all groups whose non-abelian chief factors are isomorphic to a direct product of copies of some $A$. Here, $A$ runs through all elements of $s$. In particular, if $s=\emptyset$, then $f_1(s)=\mathfrak{S}$ the formation of all finite solvable groups.
It is trivial that  formation $f_1(s)$ is extension-closed for every set $s$.\\

Next, we introduce certain group classes so that we could define special subsets $s\subset S$, which will be applied to $f_1$. In recent years, 
several theories have been proposed to study various generalized $p$-local group classes with respect to chief factors (see, e.g.,  \citep{Miao-2018},\citep{Gao-2021}). In parallel with these developments, we now introduce generalized solvable group classes with respect to maximal subgroups. There have been numerous studies to investigate finite groups in this perspective: Thompson\citep{Thompson-1968} classified minimal simple groups(non-solvable group all of whose
 proper subgroups are solvable); Guralnick\citep{Guralnick-1983} studied the structure of groups in which every maximal subgroup has prime power index. \\

In 2015, Demina and Maslova\citep{Maslova-2013}
introduced the following group class with respect to solvability and index of maximal subgroups:
$$\mathfrak{J}_{pr}=\{ G~\mid  \forall M\in Max(G),~either~M~is~ solvable~or~|G:M|~is~a~prime~power\}.$$
And the possible structure of the non-abelian composition factors of a group in $\mathfrak{J}_{pr}$ was provided\citep{Demina-2015}. {We denote the set of these simple groups as $s_1$ and define $f_1(s_1)$ as the extension formation $\hat{\mathfrak{J}_{pr}}$ associated with $\mathfrak{J}_{pr}$.}

And a more general group class was defined in \citep{Ran-2024}:
$$\mathfrak{J}=\{ G\mid  \forall M\in Max(G),~either~M~is~p-solvable~or~|G:M|~is~a~prime~power\}.$$ 
{We define $s_2$ as the collection of the corresponding non-abelian composition factors of a group in $\mathfrak{J}$ and he extension formation $\hat{\mathfrak{J}_{pr}}$ as $f(s_2)$.}\\

{Let $\mathfrak{F}$ be a formation. A group $G$ is called 
a minimal non-$\mathfrak{F}$ group if $G\notin \mathfrak{F}$ but all proper subgroups of $G$ belongs to $\mathfrak{F}$.} Recall that a group $G$ is minimal non-nilpotent\citep{Huppert-1967}  if $G$ is non-nilpotent and every proper subgroup $M$ of $G$ is nilpotent. Similarly, we define a group $G$ is minimal non-solvable if $G$ is non-solvable and every maximal subgroup $M$ of $G$ is solvable. Specifically, $A_5$ is an example.  In fact, we are motivated to consider the following group classes: 
$$ \mathfrak{F'}=\left\{
\begin{array}{clr}
       &      & \ M\ is\ p-solvable,~or\\
G\mid    &  \forall M\in Max(G),    & minimal~non-solvable,~or\\
     &      & \ |G:M|\ is\ a\ q-prime\ power
\end{array}\right\} ;$$
By localization, we get the definition of minimal non-$p$-solvable groups and define:
$$ \mathfrak{F''}=\left\{
\begin{array}{clr}
       &      & \ M\ is\ p-solvable,~or\\
G\mid    &  \forall M\in Max(G),    & minimal~non-p-solvable,~or\\
     &      & \ |G:M|\ is\ a\ q-prime\ power
\end{array} \right\} .$$
\textbf{Remark:} {$q$ is not necessarily inequal to $p$. So there are {at least, not exactly,} one case happens for every group lying in the given group class.\\
These group classes extend the group class $\mathfrak{J}_{pr}$ and the  inclusion relation 
$$\mathfrak{J}_{pr}\subset\mathfrak{J}\subset\mathfrak{F'}\subset\mathfrak{F''}$$
is illustrated by the following examples. 
\begin{example}{\rm\cite[ Example 1.1]{Ran-2024}}
{Consider $G=PSL(2,2^4)$} {of order 4080 and} suppose $p=17$. {Notably, maximal subgroup $A_5$ is non-solvable but $p$-solvable,}   illustrating that $G\notin\mathfrak{J}_{pr}$ but $G\in\mathfrak{J}$.
\end{example}
\begin{example}
{Consider} $G=A_7$ of order 2520 and suppose $p=7$. Since maximal subgroups of $PSL(2,7)$ are solvable ($Z_7:Z_3$ and $S_4$), $PSL(2,7)$ is minimal non—solvable. So $G\in \mathfrak{F}'\backslash \mathfrak{J}$.
\begin{center}
\begin{tabular}{*{6}{c}}
\hline
structure&order&{prime power index}&solvability  &p-solvability \\
\hline
 $A_6$&360  &T(7)& F&T\\
\hline
 $PSL(2,7)$&168 & F(15)&F & F\\
\hline
 $S_{5}$&120 &F(21)&F &T\\ \hline
 $(A_4\times 3): 2$&72 &F(35)&T &T\\ \hline
\end{tabular}    
\par\small \vspace{0.05cm}\textbf{Table 2:}  Maximal subgroups of $A_7$.
\end{center}
\end{example}
\begin{example}
{Consider} Sporadic Mathieu group $G=M_{11}$ of order 7920 and suppose $p$ = 11. Table 4 shows that $PSL(2,11)$ is not minimal non-solvable but minimal non-11-solvable. Hence, $G\in\mathfrak{F''}\backslash \mathfrak{F}'$.
\begin{center}
\begin{tabular}{*{5}{c}}
\hline
structure & order &Solvability& p-solvability & prime power index \\
\hline
$M_{10}$ & 720 &F& T & T(11) \\
\hline
$PSL_2{(11)}$ & 660 &F& F & F(12) \\
\hline
$M_9:S_2$ & 144 &F& T & F(55) \\
\hline
$S_5$ & 120 &F& T & F(66) \\
\hline
$Q_8:S_3$ & 48 &T& T & F(165) \\
\hline
\end{tabular}
\vspace{0.05cm}\par\small  \textbf{Table 3:} Maximal subgroups of $M_{11}$.\\
\vspace{0.25cm}
\begin{tabular}{*{3}{c}}
\hline
structure & solvability & p-solvability \\
\hline
$A_5$ & F & T \\
\hline
$11:5$ & T & T \\
\hline
$D_{12}$ & T & T \\
\hline
\end{tabular}
\vspace{0.05cm}
\par\small \textbf{Table 4:}  Maximal subgroups of $PSL_2(11)$.
\end{center}
\end{example}

\noindent Applying $f_1$ in a similar way, we define the extension formations $\hat{\mathfrak{F}'}$~and~  $\hat{\mathfrak{F}''}$ with respect to non-abelian composition factor. For later use, we also define:
$$\mathfrak{F}_1=\{G\mid \forall M\lessdot G,\ M\ is\ {p-solvable}\};$$
$$\mathfrak{F}_2=\{G\mid \forall M\lessdot G,\ M\ is\ { either~p-solvable~or~minimal~non-p-solvable}\}$$
and get extension formations $\hat{\mathfrak{F}_1}$ and $\hat{\mathfrak{F}_2}$. Such extension formations are referred to as generalized solvable extension formations and Diagram \eqref{eq:3.1} shows the relationship among them.
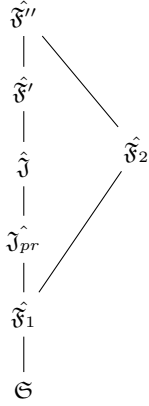
\begin{equation}\label{eq:3.1}
\begin{tikzpicture}[>=Stealth, every node/.style={font=\small}, node distance=22mm]

   nodes
  \node (Sa) at (0,0) {$\hat{\mathfrak{F}''}$};
\node (Sb) at (0,-1) {$\hat{\mathfrak{F}'}$};
\node (Sc) at (0,-2) {$\hat{\mathfrak{J}}$};
\node (Sd) at (0,-3) {$\hat{\mathfrak{J}_{pr}}$};
\node (Se) at (0,-4) {$\hat{\mathfrak{F}_{1}}$};
\node (Sf) at (0,-5) {$\mathfrak{S}$};
\node (Fe) at (1.5,-1.8) {$\hat{\mathfrak{F}_2}$};  
   horizontal tau arrows (left->middle)
  \draw[-] (Sa) -- node[above] {} (Sb);
  \draw[-] (Sb) -- node[above] {} (Sc);
  \draw[-] (Sc) -- node[above] {} (Sd);
  \draw[-] (Sd) -- node[above] {} (Se);
  \draw[-] (Se) -- node[above] {} (Sf);
  
\draw[-] (Fe) -- node[above] {} (Se);
\draw[-] (Fe) -- node[above] {} (Sa);

\end{tikzpicture}
\end{equation}

\section{Decomposition of formation morphism}
In this section, we focus on the decomposition of formation morphisms. They are the maximal subgroup functors, local formation morphism and generating formation morphism, along with various contraction and extension functors(Eg. $\mathcal{S},\mathcal{L}_1,\mathcal{L}_2, \mathcal{P}_c, \mathcal{P}_{ci}$).
\begin{definition}We use the following terminology for the morphisms that arise in Diagram \eqref{eq:1.3}.
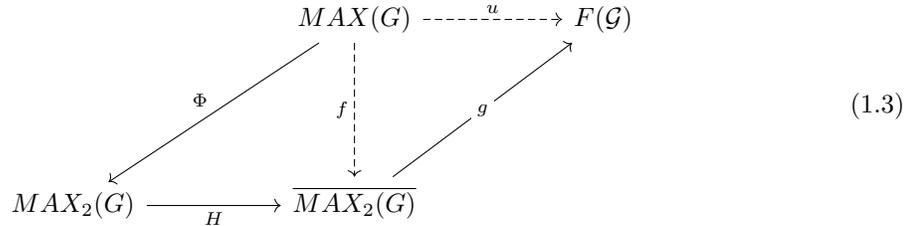
\begin{equation}
\tag{\ref{eq:1.3}}
\begin{tikzcd}[column sep=5em, row sep=5em]
  & MAX(G) \arrow[r, dashed,"u"] \arrow[d, dashed,"f"']\arrow[dl, "\Phi"'] &  F(\mathcal{G})\\
  MAX_2(G)\arrow[r, "H"'] 
  & \overline{MAX_2(G)}
  \arrow[ur, bend left=0, "g" description]
\end{tikzcd}
\end{equation}
(1). A morphism
$
u:{MAX}(G)\rightarrow\{\text{formations}\},
$
is called a {local formation morphism}.\\
(2). A morphism $g$ from  $MAX_2(G)$ to $F(\mathcal{G})$:
$g:MAX_2(G)\rightarrow F(\mathcal{G})$
with operation $g(\mathcal Y)=\{\mathfrak F\in F(\mathcal G)\mid~\text{if}~\mathcal Y=\emptyset,~then~G\in \mathfrak F\}$ is called a generating morphism.\\
(3). A morphism $f$ is a maximal subgroup functor if it assigns to a suitable maximal subgroup set $\mathcal{X}$ a possibly empty  set $f(\mathcal{X})$ of second maximal subgroups of $G$ with $g(f(\mathcal{X}))\in F(\mathcal{G})$.\\
(4). A morphism 
$\Phi:~MAX(G)\rightarrow MAX_2(G)$
with operation $\Phi(\mathcal{Y})=\{H\lessdot\lessdot G|~\exists M\in \mathcal{Y}~s.t.~ H\lessdot M\}$
is called a natural morphism.
\end{definition}
\noindent\textbf{Remark:}\\
(1). Comparing with the concept of formation function in Section 2,
the morphism $u$ offers an alternative way of assigning formations..\\
(2). The maximal subgroup functor is motivated by subgroup functor defined in \citep{Salomon-1987}.
 It assigns to each group $G$ a possibly empty set f(G) of subgroups of G satisfying $\theta(f(G))=f(\theta(G))$ for any
isomorphism $\theta$:~$G\rightarrow$ $G'$.\\
(3). {Various contraction and extension functors arise as refinements of the maximal subgroup functor, particularly in constructions related to localization.}
\begin{definition}
Let {$MAX_2(G)$} be the power set of $Max_2(G)$ and $\mathcal{Z}$ is a subset of $Max_2(G)$. {Suppose $\mathcal{Y}$ is a set of maximal subgroups and then $\Phi(\mathcal{Y})$ is the set of corresponding second maximal subgroups.} \\
(1). A contraction {functor} $T_c$  is a map with operation
$${T_c}:~MAX_2(G)\rightarrow MAX_2(G)$$
$$\Phi(\mathcal{Y})\mapsto\Phi(\mathcal{Y})\cap \mathcal{Z}\triangleq \Phi(\mathcal{Y})_\mathcal{Z}$$
and satisfies that $H(\Phi(\mathcal{Y}))=f$ is a maximal subgroup functor. In other words, Diagram \eqref{eq:1.3} holds.
{Here, $H$ is a composition of contraction functor ${T_{c_i}}$ including ${T_{c}}$ and we write $H=\prod_{i=1}^{n} T_{c_i}.$}\\
(2). An extension functor is  a  map $T_e$ with operation:
$${T_e}:~MAX_2(G)\rightarrow MAX_2(G)$$
$$\Phi(\mathcal{Y})\mapsto\Phi(\mathcal{Y})\cup \mathcal{Z}\triangleq \Phi(\mathcal{Y})^\mathcal{Z}$$
{and satisfies that} $H(\Phi(\mathcal{Y}))=f$ is a maximal subgroup functor.
In other words, Diagram \eqref{eq:1.3} holds in the case of $H=\prod_{i=1}^{n} T_{c_i}T_{e_i}.$ Here, $H$ is a composition of contraction~$\&$ extension functor ${T_{c_i}}~\&~{T_{e_i}}$ including ${T_{e}}$.

\end{definition}

{Now we introduce various contraction functors.} For instance, strictly second maximal subgroups play an essential role in exploring the group structures. In 1995, Flavell\citep{Flavell-1995} established an upper bound for the number of maximal subgroups containing a strictly second maximal subgroup. Building on this, Meng and Guo \citep{Meng-2019} proved that if  \( H \) is a second maximal subgroup of \( G \) with \( G/H_G \) is solvable, then \( |G:M| \) is divisible by \( m(G,H) - 1 \). Further extending this work, they\citep{Meng-2019+} studied WSM-groups(every second maximal subgroup is in $Max_2^*(G)$), focusing on the relation between $Max_2(G)$ and $Max_2^*(G)$. In fact,  there exists the morphism $g$ between strictly second maximal subgroups and generalized p-local group classes. The behavior of $Max_2^*(G)$ on $\mathcal{Y}_1$ helps us to define group class $S^{\#}_{p}$\citep{Wang-2022}:\\
$$g:~\mathcal{Y}_1\longmapsto formation~of~all~p-solvable~groups$$
$$g:~Max_2^*(G)\cap\mathcal{Y}_1\longmapsto S^{\#}_{p}$$
Here, we define $\mathcal{L}_1=\{H\lessdot\lessdot G|~p||H| \}$ and $\mathcal{T}_{1}=\{H\lessdot\lessdot G|~ \exists M\lessdot G~s.t~H_G=M_G\}$. Then $\mathcal{Y}_1=\mathcal{L}_1\cap \mathcal{T}_{1}$.\\

For later use, we introduce {strict functor }$S$ and localized functor $\mathcal{L}_1~\& ~\mathcal{L}_2$.
\begin{definition}
Let {$MAX_2(G)$} be the power set of $Max_2(G)$ and $Max_2^*(G)$ is a subset of $Max_2(G)$. {Suppose $\mathcal{Y}$ is a set of maximal subgroup and then $\Phi(\mathcal{Y})$ is the set of corresponding second maximal subgroup.} A strict {functor} $\mathcal{S}$ {associated with $\mathcal{Y}$} is a map with operation
$$\mathcal{S}:~MAX_2(G)\rightarrow MAX_2(G)$$
$$\Phi(\mathcal{Y})\mapsto\Phi(\mathcal{Y})\cap Max_2^*(G)\triangleq \Phi(\mathcal{Y})_{\mathcal{S}}$$
and satisfies that $H(\Phi(\mathcal{Y}))=f$ is a maximal subgroup functor. In other words, Diagram \eqref{eq:1.3} holds.
\\Here, $H$ is a composition of contraction~$\&$ extension functor ${T_{c_i}}~\&~{T_{e_i}}$ including $\mathcal{S}$. 
\end{definition}

Apart from that, Sylow theorem is a classic localization, providing detailed information about  subgroups with respect to one fixed prime divisor $p$. Recent work by Wang etal.\citep{Wang-2022} has expanded this area of research by defining the set of second maximal subgroups whose orders are divisible by $p$. More recently, Miao et al.\citep{Miao-2024} focused on the set of maximal subgroups whose orders are divisible by $p$ and consider the corresponding second maximal subgroups. The correspondence between these sets and various formations are demonstrated: 
$$g:~{L_1}\cap{Max_2^*(G)}\longmapsto G\in \hat{\mathfrak{F}},~\mathfrak{F}~is~a~formation~of~all~p-solvable~groups.$$
$$g:~{L_2}\cap{Max_2^*(G)}\longmapsto G\in \hat{\mathfrak{F}},~\mathfrak{F}~is~a~formation~of~all~solvable~groups$$
Here, ${L}_{1}=\{H\lessdot\lessdot G|~ P|~|H|\}$ and ${L}_{2}=\{H\lessdot\lessdot G \mid \exists M\in Max(G,H)~s.t.~p|~|M|\}.$
\\
\\
{Now we could define localized functors with repsect to a fixed prime $p$.}
\begin{definition}
Let {$MAX_2(G)$} be the power set of $Max_2(G)$ and ${L}_{1}=\{H\lessdot\lessdot G|~ p|~|H|\}$ is a subset of $Max_2(G)$. {Suppose $\mathcal{Y}$ is a set of maximal subgroup and then $\Phi(\mathcal{Y})$ is the set of corresponding second maximal subgroup.} A contraction {functor} $\mathcal{L}_{1}$ {associated with $\mathcal{Y}$} is a map with operation
$$\mathcal{L}_{1}:~MAX_2(G)\rightarrow MAX_2(G)$$
$$\Phi(\mathcal{Y})\mapsto\Phi(\mathcal{Y})\cap {L}_{1}\triangleq \Phi(\mathcal{Y})_{\mathcal{L}_{1}}$$
and satisfies that $H(\Phi(\mathcal{Y}))=f$ is a maximal subgroup functor. In other words, Diagram \eqref{eq:1.3} holds. Here, $H$ is a composition of contraction~$\&$ extension functor ${T_{c_i}}~\&~{T_{e_i}}$ including $\mathcal{L}_{1}$.
\end{definition}
Analogously, we could get localized functor $\mathcal{L}_{2}$ associated with $${L}_{2}=\{H\lessdot\lessdot G \mid \exists M\in Max(G,H)~s.t.~p|~|M|\}.$$

\textit{3. Property Functors}\\
Now we define property functors with respect to core relation and $|M:H|$.

\begin{definition}
Let {$MAX_2(G)$} be the power set of $Max_2(G)$ and ${P}_{c}=\{H\lessdot\lessdot G|~ \forall M\lessdot G~s.t~H_G=M_G\}$ is a subset of $Max_2(G)$. {Suppose $\mathcal{Y}$ is a set of maximal subgroup and then $\Phi(\mathcal{Y})$ is the set of corresponding second maximal subgroup.} A contraction {functor} $\mathcal{P}_c$ {associated with $\mathcal{Y}$} is a map with operation
$$\mathcal{P}_c:~MAX_2(G)\rightarrow MAX_2(G)$$
$$\Phi(\mathcal{Y})\mapsto\Phi(\mathcal{Y})\cap {P}_c\triangleq \Phi(\mathcal{Y})_{\mathcal{P}_c}$$
and satisfies that $H(\Phi(\mathcal{Y}))=f$ is a maximal subgroup functor. In other words, Diagram \eqref{eq:1.3} holds. Here, $H$ is a composition of contraction~$\&$ extension functor ${T_{c_i}}~\&~{T_{e_i}}$ including $\mathcal{P}_{c}$.
\end{definition}
Similarly, we define set 
$$\mathcal{P}_{ci}=\{H\lessdot\lessdot G|~ \forall M\lessdot G~s.t~H_G=M_G~or~|M:H|~is~not~a~prime~power\}$$  
and get property functor  $\mathcal{P}_{ci}$.\\
\\
\noindent To standardize the localization, we define the following functors in quotient form.
$$T_c\cdot \mathcal{Y}(G/L)) =\mathcal{Y}(G/L)\cap \mathcal{Z}(G/L);$$
$$T_e\cdot\mathcal{Y}(G/L) =\mathcal{Y}(G/L)\cup \mathcal{Z}(G/L);$$
$$S\cdot\mathcal{Y}(G/L)= Max_2^*(G/L)\cap \mathcal{Y}(G/L);$$
$$\mathcal{L}_1\cdot\mathcal{Y}(G/L)={L}_1(G/L)\cap \mathcal{Y}(G/L);$$
$$\mathcal{L}_2\cdot\mathcal{Y}(G/L)={L}_2(G/L)\cap \mathcal{Y}(G/L);$$
$$\mathcal{P}_c\cdot\mathcal{Y}(G/L)={P}_c(G/L)\cap \mathcal{Y}(G/L);$$
$$\mathcal{P}_{ci}\cdot\mathcal{Y}(G/L)={P}_{ci}(G/L)\cap \mathcal{Y}(G/L).$$

Here, \( \mathcal{Y}(G/L), \mathcal{Z}(G/L) \in MAX_2(G/L) \).\\

\section{Notations and Preliminaries}
Our notation are standard and all groups considered in this paper are finite. {In this section, we review the results that we need for the proofs that follow.}
\begin{definition}{\rm\cite[1.8]{Gratzer-2011}}\label{lattice}
An order ($L$; $\le$) is a lattice if $sup\{a, b\}$ and $inf\{a, b\}$ exist for all $a, b\in L$.    
\end{definition}
\begin{lemma}
Let $A$ be a nonempty set and $P(A)$ the power set of $A$. Then 
$(P(A),\le)$ is a lattice.
\end{lemma}
\begin{lemma}{\rm\cite[Theorem 3.1]{Miao-2024+}}\label{core-solvable} Let $G$ be a group. Then $G$ is solvable if and only if every second maximal subgroup of G satisfies the condition that there exists a maximal subgroup $M\in Max(G,H)$ such that $H_G < M_G.$\end{lemma}
\begin{lemma}{\rm\cite[COROLLARY 3.2]{Miao-2024+}}\label{strict-core-solvable} 
 Let $G$ be a group. Then $G$ is solvable if and only if every strictly second maximal subgroup of G satisfies the condition that there exists a maximal subgroup $M\in Max(G,H)$ such that $H_G < M_G.$\end{lemma}
\begin{lemma} {\rm\cite[ Lemma 2.13]{Wang-2022}}\label{X=HN} Let $H$ be a second maximal subgroup of a group $G$ and $X\in Max(G,H)$. Assume that $N$ is a normal subgroup of $G$ such that $N\leq X$. If $N\nleq H$, then $X=HN$.\end{lemma}
\begin{lemma}\label{L-index} 
Let $L$ be a minimal normal subgroup of $G$, $M,N$ be maximal subgroups of $G$ that do not contain $L$.
If the indices of $M, N$ in $G$ are different prime powers, then $L\in \mathfrak{F}$. Here, $$\mathfrak{F} =\{G| ~H/K~is~abelian~or~H/K~is~a~direct~product~of~ PSL(2,7)\}.$$
\end{lemma}
\begin{proof}
It is obvious that $G=ML=NL$. Assume $|G:M|=p^a$ and $|G:N|=q^b$. As $|G:M|=|ML:M|=|L:L\cap M|$, we have $|L:L\cap M|=p^a$. Similarly, $|L:L\cap M|=q^b$. Since $L\cap M < L$, there exists a maximal subgroup $L_1$ of $L$ such that $L\cap M \le L_1$. Similarly, there exists a maximal subgroup $L_2$ of $L$ such that $L\cap N \le L_2$. Then $|L:L_1|~\mid~|L:L\cap M|=p^a$ and $|L:L_2|~\mid~|L:L\cap N|=q^b$. We set $|L:L_1|=p^{a_1}$ and $|L:L_2|=q^{b_1}$. Since $L$ is a minimal normal subgroup of $G$, we set $L=A_1\times A_2\times...A_n$, here $A_i$ are isomorphic simple groups. There exist $A_i$ with $A_i\nleq L_1$ and $A_j$ with $A_j\nleq L_2$. Then $L=A_iL_1=A_jL_2$. It follows that $|L_1:L_1\cap A_i|=p^{a_1}$ and $|L_2:L_2\cap A_j|=q^{b_1}$. Then there exist maximal subgroups $B_1$ and $B_2$ of $A_1$ such that $L_1\cap A_i \le B_1$ and $L_2\cap A_j \le B_2$. Then we set $|A_i:B_1|=p^{a_2}$ and $|A_j:B_2|=q^{b_2}$. As $A_i$ and $A_j$ are isomorphic simple groups, there exists $\alpha\in Aut(G)$ such that $A_j=A_i^\alpha$. Then $|A_i:B_2^\alpha|=|A_j^\alpha:B_2^\alpha|=|A_j:B_2|$. Hence, $|A_i:B_2^\alpha|=q^{b_2}$ and $|A_i:B_1|=p^{a_2}$. By the result in \citep{Guralnick-1983},  $A_i$ is isomorphic to $PSL(2,7)$ and $L$ is a~direct~product~of~ $PSL(2,7)$.
\end{proof}

\section{\texorpdfstring{Set $\Phi(\mathcal{X})$ and its localization}%
        {Set Phi(X) and its localization}}
In this section, we primarily focus on the localization of set $\Phi(\mathcal{X})$ and consider the core relation and the index $|M:H|$. Specifically, the set $\Phi(\mathcal{X})_{\mathcal{L}_1}$ serves as the primary focus of this section and $\Phi(\mathcal{X})_{\mathcal{L}_2}$ is an extension of $\Phi(\mathcal{X})_{\mathcal{L}_1}$. 
Moreover, various extension and property functors are employed and specific formations are involved in Diagram \eqref{eq:5.1}.\\
\begin{equation}\label{eq:5.1}
\begin{tikzpicture}[scale=1.8]

  \node (Sa) at (0,0) {$\Phi(\mathcal{X})_{\mathcal{L}_1}$};
  \node (Ta) at (2,0) {$\hat{\mathfrak{F}_1}$};
   \node (Fa) at (4,0) {$\hat{\mathfrak{F}_1}$};
   \node (Ka) at (6,0) {$\Phi(\mathcal{X})_{\mathcal{L}_2}$};

  \node (Sb) at (0,-2) {$\Phi(\mathcal{X})_{(\mathcal{L}_1)_S}$};
  \node (Tb) at (2,-2) {$\hat{\mathfrak{F}_2}$};
   \node (Fb) at (4,-2) {$\hat{\mathfrak{F}_2}$};
   \node (Kb) at (6,-2) {$\Phi(\mathcal{X})_{(\mathcal{L}_2)_S}$};

\node (Sd) at (1,0.7) {$\mathcal{P}_{ci}\cdot{\Phi(\mathcal{X})}_{\mathcal{L}_1}^{E_1}$};
  \node (Td) at (3,0.7) {$\hat{\mathfrak{F}''}$};
  \node (Fd) at (5,0.7) {$\hat{\mathfrak{J}}$};
  \node (Kd) at (7,0.7) {$\mathcal{P}_{c}\cdot\Phi(\mathcal{X})_{\mathcal{L}_2}$};


  \node (Sc) at (1,-1.3) {$\mathcal{P}_{ci}\cdot{\Phi(\mathcal{X})}_{(\mathcal{L}_1)_S}^{E_1}$};
  \node (Tc) at (3,-1.3) {$\hat{\mathfrak{F}''}$};
  \node (Fc) at (5,-1.3) {$\hat{\mathfrak{F}'}$};
  \node (Kc) at (7,-1.3) {$\mathcal{P}_{ci}\cdot({\Phi(\mathcal{X})}_{\mathcal{L}_2}^{E_1})_S$};
  \draw[->] (Sa) -- node[above] {} (Ta);
  \draw[->] (Sb) -- node[above] {} (Tb);
  \draw[->] (Sc) -- node[above] {} (Tc);
  \draw[->, dashed] (Ta) -- node[above] {} (Fa);
  \draw[->] (Fa) -- node[above] {} (Ka);
  \draw[->] (Fb) -- node[above] {} (Kb);
  \draw[->] (Fc) -- node[above] {} (Kc);
  \draw[->, dashed] (Tb) -- node[above] {} (Fb);
  \draw[->, dashed] (Tc) -- node[above] {} (Fc);

\draw[->] (Sd) -- node[above] {} (Td);
\draw[->, dashed] (Td) -- node[above] {} (Fd);
\draw[->] (Fd) -- node[above] {} (Kd);

  \draw[->] (Sa) to[ right=10] node[right] {} (Sb);
  \draw[->] (Sb) to[ right=12] node[left] {} (Sc);
  \draw[->] (Sa) -- node[left] {} (Sd);
  \draw[->] (Sd) -- node[left] {} (Sc);

  \draw[->] (Tb) to[ right=10] node[above right] {} (Ta);
  \draw[->] (Tc) -- node[above] {} (Tb);
  \draw[->] (Tc) to[ left=12] node[below right] {} (Td);
  \draw[->] (Td) to[ left=12] node[below right] {} (Ta);


  \draw[->] (Fa) to[ right=10] node[right] {} (Fb);

  \draw[->] (Fb) to[ right=12] node[left] {} (Fc);
  \draw[->] (Fa) -- node[left] {} (Fd);
  \draw[->] (Fd) -- node[left] {} (Fc);

   \draw[->] (Kb) to[ right=10] node[right] {} (Ka);
   \draw[->] (Kc) to[ right=12] node[left] {} (Kb);
  \draw[->] (Ka) -- node[left] {} (Kd);
  \draw[->] (Kd) -- node[left] {} (Kc);

\end{tikzpicture}
\end{equation}
\begin{lemma} 
Given objects ${\Phi(\mathcal{X})}_{\mathcal{L}_1}\in \overline{Max_2(G)}$ and $\hat{\mathfrak{F}_1}=\{G\mid \forall M\lessdot G,\ M\ is\ {p-solvable}\}\in F(\mathcal{G})$, Diagram~\eqref{eq:1.3} induces the following morphism of objects. 
\begin{equation}\label{eq:5.2}
\begin{tikzcd}[column sep=5em, row sep=5em]
  & \mathcal{X} \arrow[r, dashed,"u"] \arrow[d, dashed,"f"']\arrow[dl, "\Phi"'] &  \hat{\mathfrak{F}_1}\\
  \Phi(\mathcal{X})\arrow[r, "\mathcal{L}_1"'] 
  & {\Phi(\mathcal{X})}_{\mathcal{L}_1}
  \arrow[ur, bend left=0, "g"] 
\end{tikzcd}
\end{equation}

\end{lemma}
\begin{proof}
It suffices to show that: If $\Phi(\mathcal{X})_{\mathcal{L}_1}=\emptyset$, then $$G\in \hat{\mathfrak{F}_1}=\{G\mid \forall M\lessdot G,\ M\ is\ {p-solvable}\}.$$
Suppose, for a contradiction, that there exists a maximal subgroup $M$ of $G$ s.t. $M$ is non-$p$-solvable. We show that it is impossible. As $p||M|$, there must exist a maximal subgroup $H$ of $M$ s.t. $M_p\le H\lessdot M$, here $M_p\in Syl_p(M)$. Then $ H\in\mathcal{L}_1$ and $\Phi(\mathcal{X})_{\mathcal{L}_1}=\emptyset$ implies that $H\notin\Phi(\mathcal{X})$. It is equivalent that every maximal subgroup containing $H$ is solvable. Then $M$ is solvable, a contradiction. 
\end{proof}
\begin{corollary}
We set $H=\mathcal{L}_2$ and $\hat{\mathfrak{F}_1}=\{G\mid \forall M\lessdot G,\ M\ is\ {p-solvable}\}.$ Then Diagram \eqref{eq:5.2} induces the following morphism of objects.

\begin{equation}\label{eq:5.3}
\begin{tikzcd}[column sep=5em, row sep=5em]
  & \mathcal{X} \arrow[r, dashed,"u"] \arrow[d, dashed,"f"']\arrow[dl, "\Phi"'] &  \hat{\mathfrak{F}_1}\\
  \Phi(\mathcal{X})\arrow[r, "\mathcal{L}_2"'] 
  & {\Phi(\mathcal{X})}_{\mathcal{L}_2}
  \arrow[ur, bend left=0, "g"] 
\end{tikzcd}
\end{equation}
\end{corollary}
The following proof leads us to introduce an extension functor $E_1$:
$$E_1= \{H\lessdot\lessdot G \mid  \exists\ M\in Max(G,H)~s.t.~|G:M|\neq p^{\alpha}\}.$$

\begin{theorem}\label{a-index}
Given objects ${\Phi(\mathcal{X})}_{\mathcal{L}_1}\in \overline{Max_2(G)}$ and $\hat{\mathfrak{F}_1}\in F(\mathcal{G})$, Diagram~\eqref{eq:5.2} admits a completion into the following diagram.
\end{theorem}
\begin{equation}\label{eq:5.4}
\begin{tikzcd}[column sep=5em, row sep=5em]
  & \mathcal{X} \arrow[r, dashed,"u"] \arrow[d, dashed,"f"']\arrow[dl, "\Phi"'] &  \hat{\mathfrak{F}_1}\arrow[r, ""'] &  \hat{\mathfrak{F}''}\\
  \Phi(\mathcal{X})\arrow[r, "\mathcal{L}_1"'] 
  & {\Phi(\mathcal{X})}_{\mathcal{L}_1}
  \arrow[r, "E_1","\mathcal{P}_{ci}"'] \arrow[ur, bend left=0, "g" ] 
  & \mathcal{P}_{ci}\cdot{\Phi(\mathcal{X})}_{\mathcal{L}_1}^{E_1} \arrow[ur, bend left=0, "g" ]
\end{tikzcd}
\end{equation}

\begin{proof}
Obviously, $G$ is not simple and suppose $L$ is a minimal normal subgroup of $G$. Working by induction on $|G|$, we consider the quotient group $G/L$.\\
\\
Now we will show that $G/L\in \hat{\mathfrak{F}''}$. Pick any $H/L \in {\Phi(\mathcal{X}(G/L))}_{\mathcal{L}_1}^{E_1}=\emptyset$. As $\exists\ M_1/L\in Max(G/L,H/L)$~\\s.t.~$|G/L:(M_1)/L|\neq p^{\alpha}$, then $\ M_1\in Max(G,H)$ and $|G:M_1|\neq p^{\alpha}$. As $H/L \in \Phi(\mathcal{X}(G/L))_{\mathcal{L}_1}$, we have $\exists\ M_2/L\in Max(G/L,H/L)$ s.t. $M_2/L$ is~non-solvable and $p||H/L|$.  Then we have $ M_2\in Max(G,H)$ and $M_2$ is non-solvable. It leads that $H \in {\Phi(\mathcal{X})}_{\mathcal{L}_1}^{E_1}$. It is obvious that $H\notin \mathcal{P}_{ci}$ given that $H\notin \mathcal{P}_{ci}$. Applying induction, $G/L\in \hat{\mathfrak{F}''}$.\\
\\
We will show that $G/L$ is solvable and $L$ is unique. Assume that $L$ is non-$p$-solvable, which gives that  $p||L|$. Pick any maximal subgroup $M/L$ of $G/L$ and $H/L$ of $M/L$, then we have $H\in\Phi(\mathcal{X})$ and $p||H|$. So $H\in{\Phi(\mathcal{X})}_{\mathcal{L}_1}$. It is equivalent to say that $\exists M\in Max(G,H)~s.t.~H_G<M_G$ for every second maximal subgroup $H/L$ of $G/L$. By Lemma \ref{core-solvable}, $G/L$ is solvable. We claim that $L$ is unique. Suppose, for a contradiction, $L_1$ and $L_2$ are two distinct minimal normal subgroups of $G$, then both $G/L_1$ and $G/L_2$ are solvable {and so is} $G/L_1\cap L_2$. Consequently, $G$ is solvable. This is a contradiction.\\
\\
To show $G\in \hat{\mathfrak{F}''}$, it remains us to consider any maximal subgroup $M$ of $G$. We first consider the case $L\leq M$. Then we have $|G:M|=|G/L:M/L|$ and $|G:M|$ is a prime power. Now we consider the case $L\nleq M$. In this case, $G=LM$ and $M_G=1$. There is nothing to prove if $M$ is $p$-solvable. So we assume that $M$ is non-$p$-solvable. Pick any maximal subgroup $H$ of $M$, we claim that $p\nmid|H|$ and $H$ is $p$-solvable. Otherwise, we have $p||H|$, then $H\in\Phi(\mathcal{X})_{\mathcal{L}_1}$. It concludes that $\exists M_3\in Max(G,H)~s.t.~ H_G<(M_3)_G$ and $|M_3:H|=q^{\beta}$ is a prime power. By Lemma \ref{X=HN}, we get $|M_3:H|=|HL:H|=|L:L\cap H|=q^{\beta}$. Using Frattini argument, we have $G=LN_G(L_q)=LM_4$, here $Q\le N_G(L_q)\le M_4$ and $Q\in Syl_q(G)$. If $|G:M_4|={q_1}^{\beta_1}$, then $|L:L\cap M_4|=|LM_4:L|=|G:M_4|={q_1}^{\beta_1}$. Here we get $|L:L\cap H|=q^{\beta}$ and $|L:L\cap M_4|={q_1}^{\beta_1}$. This is a contradiction by Lemma \ref{L-index}. Now we consider the case $|G:M_4|\neq{q_1}^{\beta_1}$. In this case, $H_1\in {\Phi(\mathcal{X})}_{\mathcal{L}_1}^{E_1}$  for every maximal subgroup $H_1$ of $M_4$. Then $\exists M_5\in Max(G,H)~s.t. (H_1)_G<(M_5)_G$ and $|M_5:H_1|=(q_2)^{\beta_2}$ is a prime power. By Lemma \ref{X=HN}, we get $|M_5:H_1|=|H_1L:H_1|=|L:L\cap H_1|=(q_2)^{\beta_2}$. Now we get $|L:L\cap H|=q^{\beta}$ and $|L:L\cap H_1|=(q_2)^{\beta_2}$. This is also a contradiction by Lemma \ref{L-index}.\\
\\
Now for any arbitrary maximal subgroup $M$ of $G$, we have either $M$ is $p$-solvable/minimal~non-$p$-solvable or $|G:M|$ is a prime power. Hence, $G\in \hat{\mathfrak{F}''}$ and such a completion exists..
\end{proof}

\begin{theorem}\label{}
We set $H=\mathcal{L}_2$ and $\hat{\mathfrak{J}}=\{G\mid \forall M\lessdot G,\ either\ M\ is\ p-solvable\ or\ |G:M|\ is\ a\ prime\ power\}.$ Substituting $\mathcal{P}_{ci}$ to $\mathcal{P}_{c}$, Diagram~\ref{eq:5.3} admits a completion
into the following diagram. 
\begin{equation}\label{eq:5.5}
\begin{tikzcd}[column sep=5em, row sep=5em]
  & \mathcal{X} \arrow[r, "u"] \arrow[d, dashed,"f"']\arrow[dl, "\Phi"'] &  \hat{\mathfrak{F}_1}\arrow[r, ""'] &  \hat{\mathfrak{J}}\\
  \Phi(\mathcal{X})\arrow[r, "\mathcal{L}_2"'] 
  & {\Phi(\mathcal{X})}_{\mathcal{L}_2}
  \arrow[r, "\mathcal{P}_{c}"] \arrow[ur, bend left=0, "g" ] 
  & \mathcal{P}_{c}\cdot{\Phi(\mathcal{X})}_{\mathcal{L}_2}\arrow[ur, bend left=0, "g" ]
\end{tikzcd}
\end{equation}
\end{theorem}
\begin{proof}
Work by induction on $|G|$. Obviously, $G$ is not simple and suppose $L$ is a minimal normal subgroup of $G$. We consider the quotient group $G/L$ and show that $G/L\in \hat{\mathfrak{J}}$. Pick any $H/L \in \Phi(\mathcal{X}(G/L))_{\mathcal{L}_2}$. Similar to 
the second paragraph in the proof of Theorem \ref{a-index}, we see that $H \in \Phi(\mathcal{X})_{\mathcal{L}_2}$. Applying induction, $G/L\in \hat{\mathfrak{J}}$.\\
\\
A similar argument works for that $G/L$ is solvable and $L$ is unique.\\
\\
To show $G\in \hat{\mathfrak{J}}$, it remains us to consider any maximal subgroup $M$ of $G$. We first consider the case $L\leq M$. Then we have $|G:M|=|G/L:M/L|$ and $|G:M|$ is a prime power. Now we consider the case $L\nleq M$. In this case, $G=LM$ and $M_G=1$. There is nothing to prove if $M$ is $p$-solvable. So we assume that $M$ is {non-$p$-solvable}. Pick any maximal subgroup $H$ of $M$, then $H\in \Phi(\mathcal{X})_{\mathcal{L}_2}$. It concludes that $\exists M_6\in Max(G,H)~s.t.~ H_G<(M_6)_G$ and $M_6=HL$ by Lemma \ref{X=HN}. Then we get {$L\cap M\le \Phi(M)$} and $L\cap M$ is solvable. As $G/L=ML/L\cong L/L\cap M$, $M$ is solvable. This is a contradiction.\\
\\
Now for any arbitrary maximal subgroup $M$ of $G$, we have $M$ is $p$-solvable or $|G:M|$ is a prime power. Hence, $G\in \hat{\mathfrak{J}}$.
\end{proof}

\begin{lemma}
Set $H=(\mathcal{L}_1)_S$ and $\hat{\mathfrak{F}_2}=\{G\mid \forall M\lessdot G,\ M\ is\ { either~p-solvable~or~minimal~non-p-solvable}\}.$ Diagram~\eqref{eq:1.3} induces the following morphism of objects. 
\begin{equation}\label{eq:5.6}
\begin{tikzcd}[column sep=5em, row sep=5em]
  & \mathcal{X} \arrow[r, dashed,"u"] \arrow[d, dashed,"f"']\arrow[dl, "\Phi"'] &  \hat{\mathfrak{F}_2}\\
  \Phi(\mathcal{X})\arrow[r, "(\mathcal{L}_1)_S"',""] 
  & {\Phi(\mathcal{X})}_{(\mathcal{L}_1)_S}
  \arrow[ur, bend left=0, "g"] 
\end{tikzcd}
\end{equation}
\end{lemma}
\begin{proof}
It is trivial if $M$ is $p$-solvable for any maximal subgroup $M$ of $G$. Assume that $M$ is non-$p$-solvable and $p\mid|M|$. Then we have either $p\nmid|H|$ or $p\mid|H|$ for any maximal subgroup $H$ of $M$. In the former case, $H$ is $p$-solvable. In the latter case, we claim that $H\notin Max_2^*(G)$. Otherwise,  $H\notin\Phi(\mathcal{X})$. It gives $M$ is solvable, which is a contradiction. As $H\notin Max_2^*(G)$, then there exists a strictly second maximal subgroup $H^{M_7}$ s.t. $H<...<H^{M_7}\lessdot M_7\lessdot G$. It leads to $H^{M_7}\notin\Phi(\mathcal{X})$ and $M_7$ is solvable. So $H$ is solvable in the latter case. Hence, For any maximal subgroup $M$ of $G$, $M$ is $p$-solvable or minimal non-$p$-solvable.\\
\end{proof}

\begin{corollary}\label{SXB-e}
Set $H=(\mathcal{L}_2)_S$ and  Diagram~\eqref{eq:5.6} induces the following morphism of objects. 
\begin{equation}\label{eq:5.7}
\begin{tikzcd}[column sep=5em, row sep=5em]
  & \mathcal{X} \arrow[r, dashed,"u"] \arrow[d, dashed,"f"']\arrow[dl, "\Phi"'] &  \hat{\mathfrak{F}_2}\\
  \Phi(\mathcal{X})\arrow[r, "(\mathcal{L}_2)_S"',""] 
  & {\Phi(\mathcal{X})}_{(\mathcal{L}_2)_S}
  \arrow[ur, bend left=0, "g"] 
\end{tikzcd}
\end{equation}
\end{corollary}

\begin{theorem}\label{}
Diagram~\eqref{eq:5.6} admits a completion
into the following diagram. 
\begin{equation}\label{eq:5.8}
\begin{tikzcd}[column sep=5em, row sep=5em]
  & \mathcal{X} \arrow[r, dashed,"u"] \arrow[d, dashed,"f"']\arrow[dl, "\Phi"'] &  \hat{\mathfrak{F}_2}\arrow[r, ""] &\hat{\mathfrak{F}''}\\
  \Phi(\mathcal{X})\arrow[r, "(\mathcal{L}_1)_S"',""] 
  & {\Phi(\mathcal{X})}_{(\mathcal{L}_1)_S}
  \arrow[ur, bend left=0, "g"] \arrow[r, "E_1","\mathcal{P}_{ci}"']&
\mathcal{P}_{ci}\cdot{\Phi(\mathcal{X})}_{(\mathcal{L}_1)_S}^{E_1} \arrow[ur, bend left=0, "g"]
\end{tikzcd}
\end{equation}

\end{theorem}
\begin{proof}
Now we turn to prove the statement of the theorem.
Obviously, $G$ is not simple. Now we proceed to prove by induction on $|G|$. Suppose $L$ is a minimal normal subgroup of $G$ and we consider the quotient group $G/L$.\\
\\
Now we will show that $G/L\in \hat{\mathfrak{F}''}$. Pick any $H/L \in {\Phi(\mathcal{X}(G/L))}_{(\mathcal{L}_1)_S}^{E_1}$, then $H\notin \mathcal{P}_{ci}(G/L)$. By proof of Theorem \ref{a-index} and $H/L \in Max_2^*(G/L)$, then $H\in Max_2^*(G)$ and $H\in {\Phi(\mathcal{X})}_{(\mathcal{L}_1)_S}^{E_1}$. It is clear that $H\notin \mathcal{P}_{ci}$ Hence, we get $G/L\in \hat{\mathfrak{F}''}$ by induction.\\
\\
A similar argument works for that $G/L$ is solvable and $L$ is unique.\\
\\
To show $G\in \hat{\mathfrak{F}''}$, it remains us to consider any maximal subgroup $M$ of $G$. We first consider the case $L\leq M$. Then we have $|G:M|=|G/L:M/L|$ and $|G:M|$ is a prime power. Now we consider the case $L\nleq M$. In this case, $G=LM$ and $M_G=1$. There is nothing to prove if $M$ is $p$-solvable. So we assume that $M$ is non-$p$-solvable. Pick any maximal subgroup $H$ of $M$, we claim that $p\nmid|H|$ and $H$ is $p$-solvable. Otherwise, we have $p||H|$, then we consider the following subcases.\\
\\
Subcase 1: $H\in Max_2^*(G)$\\
    It concludes that $\exists M_{8}\in Max(G,H)~s.t.~ H_G<(M_{8})_G$ and $|M_{8}:H|={q_3}^{\beta_3}$ is a prime power. By Lemma \ref{X=HN}, we get $|M_{8}:H|=|HL:H|=|L:L\cap H|={q_3}^{\beta_3}$. Using Frattini argument, we have $G=LN_G(L_{q_3})=LM_{9}$, here $N_G(L_{q_3})\le M_{9}$. If $|G:M_{9}|={q_4}^{\beta_4}$, then $|L:L\cap M_{9}|=|LM_{9}:L|=|G:M_{9}|={q_4}^{\beta_4}$. Here we get $|L:L\cap H|={q_3}^{\beta_3}$ and $|L:L\cap M_{9}|={q_4}^{\beta_4}$. This is a contradiction by Lemma \ref{L-index}. Now we consider the case $|G:M_{9}|\neq{q_4}^{\beta_4}$. Pick $(M_{9})_{q_3}\le H_2\lessdot M_{9}$. It follows that $\exists M_{10}\in Max(G,H_2)~s.t. (H_2)_G<(M_{10})_G$ and $|M_{10}:H_2|=(q_5)^{\beta_5}$ is a prime power. By Lemma \ref{X=HN}, we get $|M_{10}:H_2|=|H_2L:H_2|=|L:L\cap H_2|=(q_5)^{\beta_5}$. Now we get $|L:L\cap H_2|=(q_3)^{\beta_3}$ and $|L:L\cap H_2|=(q_5)^{\beta_5}$. This is also a contradiction by Lemma \ref{L-index}.\\
\\
Subcase 2: $H\notin Max_2^*(G)$ \\
There exists a strictly second maximal subgroup $H^{M_{11}}$
s.t. $H\le...\le H^{M_{11}}\lessdot M_{11}$. It is obvious that $p||H^{M_{11}}|$. If $M_{11}$ is solvable. then $H$ is solvable. If $M_{11}$ is not solvable.
It concludes that $\exists M_{12}\in Max(G,H^{M_{11}})~s.t.~ (H^{M_{11}})_G<(M_{12})_G$ and $|M_{12}:H^{M_{11}}|={q_6}^{\beta_6}$ is a prime power. By Lemma \ref{X=HN}, we get $|M_{12}:H^{M_{11}}|=|H^{M_{11}}L:H^{M_{11}}|=|L:L\cap H^{M_{11}}|={q_6}^{\beta_6}$. Using Frattini argument, we have $G=LN_G(L_{q_6})=LM_{13}$, here $N_G(L_{q_6})\le M_{13}$. The contradiction follows with a similar argument.

Now for any arbitrary maximal subgroup $M$ of $G$, we have $M$ is $p$-solvable/minimal~non-$p$-solvable or $|G:M|$ is a prime power. Hence, $G\in \hat{\mathfrak{F}''}$.
\end{proof}

By {modifying} the set $\mathcal{P}_{ci}\cdot{\Phi(\mathcal{X})}_{(\mathcal{L}_1)_S}^{E_1}$ to  $\mathcal{P}_{ci}\cdot({\Phi(\mathcal{X})}_{\mathcal{L}_2}^{E_1})_S $, we get the following corollary.

\begin{corollary}\label{}
Diagram~\eqref{eq:5.7} admits a completion
into the following diagram.

\begin{equation}\label{eq:5.9}
\begin{tikzcd}[column sep=5em, row sep=5em]
  & \mathcal{X} \arrow[r, dashed,"u"] \arrow[d, dashed,"f"']\arrow[dl, "\Phi"'] &  \hat{\mathfrak{F}_2}\arrow[r, ""] &\hat{\mathfrak{F}'}\\
  \Phi(\mathcal{X})\arrow[r, "\mathcal{L}_2"',""] 
  & \Phi(\mathcal{X})_{(\mathcal{L}_2)_{S}} 
  \arrow[ur, bend left=0, "g"] \arrow[r, "\mathcal{P}_{ci}"']&
\mathcal{P}_{ci}\cdot({\Phi(\mathcal{X})}_{\mathcal{L}_2}^{E_1})_S\arrow[ur, bend left=0, "g"]
\end{tikzcd}
\end{equation}

\end{corollary}
\begin{proof}
The proof is similar. But for the case  $|G:M_{12}|\neq q_4^{\beta_4}$ in the first subcase, we need to consider $H_2\in Max_2^*(G)$ or $H_2\notin Max_2^*(G)$. In the fisrt case, $H_2\in ({\Phi(\mathcal{X})}_{\mathcal{L}_2}^{E_1})_S $ and the proof is similar. In the second case, there exists a strictly second maximal subgroup $H^{M_{14}}$
s.t. $H_2\le...\le (H_2)^{M_{14}}\lessdot M_{14}$. Now we have either $(M_{14})_G=1$ or  $(M_{14})_G\neq 1$. In the former case, we get $|L:L\cap M_{14}|=q_7^{\beta_7}$ if $|G:M_{14}|=q_7^{\beta_7}$, a contradiction. Otherwise, $|G:M_{14}|\neq q_7^{\beta_7}$. Then $H_2\in ({\Phi(\mathcal{X})}_{\mathcal{L}_2}^{E_1})_S $ and the proof is similar. In the latter case, if $M_{14}$ is $p$-solvable, then $L$ is $p$-solvable. This is a contradiction. If $M_{14}$ is non-$p$-solvable, then $p||M_{14}|$ and $H^{M_{14}}\in ({\Phi(\mathcal{X})}_{\mathcal{L}_2}^{E_1})_S $. The contradiction is similar.
The similar proof works for the subcase 2.    
\end{proof}
\section{Applications}
The order-reversing map is common in Algebra. Notable instances include the Galois correspondence in galois theory and the antitone relationship between an affine variety and its coordinate ring in commutative algebra. In this paper, the morphism $g$ is \emph{order-reversing}: For \(\mathcal{X}_1, \mathcal{X}_2 \in MAX_2(G)\),
\[
\mathcal{X}_1 \subseteq \mathcal{X}_2 \quad \Longrightarrow \quad g(\mathcal{X}_1)\supseteq g(\mathcal{X}_2).
\]
Especially, we equip both dom($g$) and cod($g$) with two operations $\wedge~\&~\vee$. On dom($g$), the operations $\wedge$ and $\vee$ are defined as set-theoretic intersection and union, respectively. On cod($g$), we define
$\mathfrak{F}_1\wedge\mathfrak{F}_2$ as the intersection of the formation $\mathfrak{F}_1$ and $\mathfrak{F}_2$; $\mathfrak{F}_1\vee\mathfrak{F}_2$ as the the formation product $\mathfrak{F}_1\mathfrak{F}_2=\{G|G~\text{is~a~group~and~}G^{\mathfrak{F}_2}\in\mathfrak{F}_1\}$. Then the morphism 
$g$ behaves like an anti-homomorphism on certain elements. 

\begin{example}
Set $a=\Phi(\mathcal{X})_{\mathcal{L}_1}$ and $b=\Phi(\mathcal{X})_{(\mathcal{L}_1)_\mathcal{S}}$. Then $g(a)=\hat{\mathfrak{F}_1}$ and    $g(b)=\hat{\mathfrak{F}_2}$. As
$$g(a\wedge b)=g(b)=\hat{\mathfrak{F}_2}=\hat{\mathfrak{F}_1}\hat{\mathfrak{F}_2}=g(a)\vee g(b)$$ and 
$$g(a\vee b)=g(a)=\hat{\mathfrak{F}_1}=\hat{\mathfrak{F}_1}\cap\hat{\mathfrak{F}_2}=g(a)\wedge g(b),$$
we get
$$g(a\wedge b)=g(a)\vee g(b);$$
$$g(a\vee  b)=g(a)\wedge g(b).$$
\end{example}
\begin{example}
Set $a=\Phi(\mathcal{X})_{\mathcal{L}_1}$ and $b=\mathcal{P}_{ci}\cdot\Phi(\mathcal{X})_{(\mathcal{L}_1)}^{E_1}$. Then $g(a)=\hat{\mathfrak{F}_1}$ and    $g(b)=\hat{\mathfrak{F}''}$. As
$$g(a\wedge b)=g(b)=\hat{\mathfrak{F}''}=\hat{\mathfrak{F}_1} \hat{\mathfrak{F}''}=g(a)\vee g(b)$$ and
$$g(a\vee b)=g(a)=\hat{\mathfrak{F}_1}=\hat{\mathfrak{F}_1}\cap\hat{\mathfrak{F}''}=g(a)\wedge g(b),$$
we get
$$g(a\wedge b)=g(a)\vee g(b);$$
$$g(a\vee b)=g(a)\wedge g(b).$$
\end{example}

\end{document}